\documentclass[12pt, reqno]{amsart}
\usepackage{amssymb,latexsym,amsmath,amsfonts, amsthm}
\usepackage{latexsym}
\usepackage{graphicx}
\usepackage{color}
\usepackage[mathscr]{eucal}
\usepackage{comment}
\usepackage[linkcolor=blue, citecolor=black]{hyperref}
\hypersetup{colorlinks=true, urlcolor=blue}
\usepackage[nameinlink]{cleveref}

\numberwithin{equation}{section}

\theoremstyle{definition}
\newtheorem{definition}{Definition}[section]
\newtheorem{example}[definition]{Example}

\theoremstyle{remark}
\newtheorem{remark}[definition]{Remark}
\theoremstyle{plain}
\newtheorem{theorem}[definition]{Theorem}
\newtheorem{lemma}[definition]{Lemma}

\newtheorem{result}[definition]{Result}

\newtheorem{corollary}[definition]{Corollary}

\newcommand{\al}{\alpha}

\newcommand{\bdy}{\partial}

\newcommand{\smoo}{\mathcal{C}}
\newcommand{\hol}{\mathcal{O}}
\newcommand{\poly}{\mathscr{P}}

\newcommand\hull[1]{\widehat{#1}}

\newcommand{\cplx}{\mathbb{C}}

\newcommand{\Gr}{{\sf Gr}}

\begin{document}

	\title[Uniform algebras and distinguished varieties]{Uniform algebras and distinguished varieties}
	\author{Sushil Gorai and Golam Mostafa Mondal}
	\address{Department of Mathematics and Statistics, Indian Institute of Science Education and Research Kolkata,
		Mohanpur -- 741 246}
	\email{sushil.gorai@iiserkol.ac.in, sushil.gorai@gmail.com}
	
	\address{Department of Mathematics, Indian Institute of Science Education and Research Pune, Pune -- 411 008}
	\email{golammostafaa@gmail.com, golam.mondal@acads.iiserpune.ac.in}
	\thanks{}
	
	\begin{abstract}
		In this article, we point out the connections between the distinguished varieties introduced by Agler and McCarthy with certain uniform algebras on bidisc studied by Samuelsson and Wold. We also prove analogues of Samuelsson-Wold result for the domains in $\mathbb{C}^2$ that are the images of the bidisc under certain proper polynomial map on $\mathbb{C}^2$. We also give a description of polynomial convex hull of graph of anti-holomorphic polynomial over the distinguished boundary of such domains. We mention the case for the symmetrized bidisc as an example. 
	\end{abstract}

	\maketitle
	\begin{center}		
		\date{}
	\end{center}
	
	%\tableofcontents	

	\section{Introduction}\label{S:intro}
	This article connects the theory of distinguished varieties--a well-explored topic in operator theory, with the notions of uniform algebra generated by holomorphic polynomials and certain pluriharmonic functions. The latter one is also a very well-studied object in several complex variables. In particular, we observe that the failure of uniform approximation for all continuous functions on the distinguished boundary of certain domains in $\mathbb{C}^2$ by elements of holomorphic polynomials in $z_1$ and $z_2,$ and some pluriharmonic functions is the presence of certain distinguished variety in the domain where the pluriharmonic functions become holomorphic. Before making these precise, let us briefly mention the theory of distinguished varieties and the theory of uniform algebras one by one.

	\smallskip
	
	In a seminal paper \cite{AglerMcCarthy2005}, Agler and McCarthy  introduced the notion of distinguished variety in the bidisc $\mathbb{D}^2$ as follows: A non-empty set $V$ in $\mathbb{C}^2$ is said to be a \textit{distinguished variety} if there exists a polynomial $p$ in $\mathbb{C}[z, w]$ such that
	\[
	V = \{(z, w) \in \mathbb{D}^2 : p(z, w) = 0\}
	\]
	and such that
	\begin{align}\label{E:Distinguishd Variety}
		\overline{V} \cap \partial\mathbb{D}^2 = \overline{V} \cap \mathbb{T}^2.
	\end{align}
	Here, $\partial \mathbb{D}^2$ represents the boundary of the $\mathbb{D}^2$, and $ \mathbb{T}^2$ is the distinguished boundary of $\mathbb{D}^2$. A distinguished variety is an algebraic variety that exits the bidisc through the distinguished boundary. The set $\overline{V}$ is the closure of $V$ within $\overline{\mathbb{D}}^2$. We will use $\partial V$ to denote the set described by (\ref{E:Distinguishd Variety}). From a topological standpoint, $\partial V$ represents the boundary of $V$ within the zero set of $p$ instead of encompassing the entirety of $\mathbb{C}^2$.

	\par One of the fundamental results in operator theory, known as And\^{o}'s inequality \cite{Ando1963}, establishes that when $T_1$ and $T_2$ are commuting operators, each with a norm not exceeding 1, the following inequality holds for any two-variable polynomial $p$:
	\begin{align}\label{E:Ando_Inequality}
		\|p(T_1, T_2)\| \leq \|p\|_{\mathbb{D}^2}
	\end{align}
	holds. Agler and McCarthy \cite[Theorem 3.1]{AglerMcCarthy2005} gave the following improvement of the inequality (\ref{E:Ando_Inequality}): if $T_1$ and $T_2$ are matrices, then
	\[
	\|p(T_1, T_2)\| \leq \|p\|_V,
	\]
	where $V$ is a distinguished variety that depends on $T_1$ and $T_2$. Additionally, in \cite[Theorem 1.12]{AglerMcCarthy2005}, the authors have shown that all distinguished varieties in the bidisc can be expressed as
	\[
	\{(z, w) \in \mathbb{D}^2 : \det(\Psi(z) - wI) = 0\},
	\]
	where $\Psi$ is an analytic matrix-valued function defined on the disk, and it is unitary on $\partial\mathbb{D}$. A similar description of distinguished varieties in the symmetrized bidisc is given in \cite{PalShalit2014}.
	
	\smallskip
	
	\par Consider a compact subset $K$ of $\mathbb{C}^n$. The space of all continuous complex-valued functions on $K$ is denoted as $\mathcal{C}(K)$, equipped with the norm $|g| = \sup_{K}|g(z)|$. We denote the closure of the set of polynomials in $\mathcal{C}(K)$ as $\mathcal{P}(K)$. For a collection of functions $g_1, \ldots, g_N \in \mathcal{C}(K)$, we use $[g_1, \ldots, g_N; K]$ to represent the uniform algebra generated by $g_1, \ldots, g_N$ on $K$. We define the set $X={(g_1(z), \ldots, g_N(z)): z\in K}$ associated with the uniform algebra $[g_1, \ldots, g_N; K]$. Using the Stone and Weierstrass theorem, we assert that
	\begin{align*}
		[ g,\cdots , g_{N}; K] = \smoo(K)
	\end{align*}
	if and only if $\mathcal{P}(X)=\mathcal{C}(X)$ and the generators $g_1, \ldots, g_N$ separate points on $K$.
	
	\smallskip
	
	\noindent If we consider $\mathcal{P}(K)$ and $\mathcal{C}(K)$ as Banach algebras, the equality $\mathcal{P}(K)=\mathcal{C}(K)$ implies the equality of their corresponding maximal ideal spaces. The maximal ideal space of $\mathcal{C}(K)$ corresponds to $K$, and that of $\mathcal{P}(K)$ corresponds to $\widehat{K}$, where $\widehat{K}$ is the polynomial convex hull of $K$ (see \cite{Gamelin}). Here, the \textit{polynomial convex hull} of $K$ is denoted as $\widehat{K}$ and is defined as follows:	
	
	\begin{align*}
		\hull K:=\left\{\al\in\cplx^n: |p(\al)|\le\max_{K}|p|~~\forall p\in \cplx[z_1,z_2,\cdots,z_n]\right\}.
	\end{align*}
	We say $K$ is \textit{polynomially convex} when $\widehat{K}= K$. Thus, polynomial convexity serves as a necessary condition for all compacts $K$ where $\mathcal{P}(K)=\mathcal{C}(K)$ holds.
	
	\smallskip
	
	\noindent Recall that an \emph{analytic disc}\index{Analytic disc} in $\mathbb{C}^n$ is a holomorphic map $\phi:\mathbb{D}\rightarrow\mathbb{C}^n$ which is non-constant and continuous on $\overline{\mathbb{D}}$. Let $K\subset\cplx^n.$ We say an analytic disc $\phi$ is present in $K$ if $\phi(\mathbb{D})\subset K.$
	In view of Lavrentiev's \cite{Lv} result, if $K$ be a compact subset of $\cplx,$ then $\poly(K)=\smoo(K)$ if and only if $K = \widehat{K}$ and there does not exist any analytic disc in $K.$ But this is far from being a sufficient condition for polynomially convex compacts in higher dimensions. This article discusses some results in which the presence of an analytic disc is the only obstruction for polynomially convex compact $K$ for which $\poly(K)=\smoo(K).$ We now talk about the Wermer maximality theorem.
	\begin{comment}
	\par  In the case of complex plane, polynomial convexity is equivalent to a simple topological condition: $K = \hat{K}$
	if and only if $\cplx\setminus K$ is connected. Another necessary condition for $\poly(K)=\smoo(K)$ is that $int(K)=\emptyset.$ For the case of complex plane, Lavrentiev \cite{Lv} showed that this two necessary conditions are also sufficient:
	for $K\subset \cplx,$ $\poly(K)=\smoo(K)$ if and only if $\hull{K}=K$ and $int(K)=\emptyset.$ 
	Things are much more complicated in higher dimensions. Polynomial convexity is not equivalent to a topological criterion in higher dimensions. For example some arcs are polynomially convex and some are not. So it is natural to consider the problem for particular classes of compact sets. For instance, the fact that the compact  $K$ lies in a submanifold of $\cplx^{n}$ simplifies the problem a bit. Let $K$ lies in a smooth real submanifold of $\cplx^{n}$ and $E$ is the collection of points where the manifold has complex tangents. Then, the size and nature of the set $E$ are the crucial factors for the equality of $\poly(K)=\smoo(K)$.
	\end{comment}
	Let $\mathbb{T}^{1}$ be the unit circle in the complex plane and $\smoo(\mathbb{T}^{1})$ be the set of all continuous complex-valued functions on $\mathbb{T}^{1}.$ Let $\mathcal{A}$ denote the set of all $f\in \smoo(\mathbb{T}^{1})$ which are boundary values of functions holomorphic on $\mathbb{D}$ and continuous on $\overline{\mathbb{D}}.$ In \cite{ZL52}, the following question was asked:	$$\textit{if } g\in \smoo(\mathbb{T}^{1})\setminus \mathcal{A}, \textit{ does the closed algebra generated by } g \textit{ and } \mathcal{A} \textit{ equal } \smoo(\mathbb{T}^{1})?$$
	
	\noindent In \cite{ZL52}, it is shown that if $g$ is real-valued or if $g$ satisfies a Lipschitz condition, the algebra generated by $g$ and $\mathcal{A}$ equals $\smoo(\mathbb{T}^{1}).$ Wermer \cite{Wer53} settled this question by proving the following:
	
	\begin{result}[Wermer]\label{R:WMT}
		If $\mathcal{B}$ is any closed subalgebra of $\smoo(\mathbb{T}^{1})$ with $\mathcal{A}\subset \mathcal{B}\subset \smoo(\mathbb{T}^{1}).$ Then either $\mathcal{A}=\mathcal{B}$ or $\mathcal{B}=\smoo(\mathbb{T}^{1}).$	
	\end{result}
	
	\noindent A uniform algebra $\mathcal{U}$ defined on a compact subset $K$ is said to be a \textit{maximal subalgebra} of $\smoo(K)$ if, for any other subalgebra $\mathcal{B}$ of $\smoo(K)$ such that $\mathcal{U}\subset \mathcal{B}\subset \smoo(K)$, it holds that either $\mathcal{U}=\mathcal{B}$ or $\mathcal{B}=\smoo(K)$. \Cref{R:WMT} is known as the \textit{Wermer Maximality Theorem.} A similar related result due to Wermer is the following \cite{Wer65}: Let $g\in C^{1}(\overline{\mathbb{D}}).$ Assume that the graph $\Gr_{\overline{\mathbb{D}}}(g)\subset\cplx^{2}$ of $g$ is polynomially convex. Let $E:=\{z\in\overline{\mathbb{D}}:\frac{\partial g}{\partial\bar{z}}(z)=0\}.$ Then 
	\begin{align*}
		[z,g;\overline{\mathbb{D}}]=\{f\in C(\overline{\mathbb{D}}):f|_{E}\in \hol(E)\}.
	\end{align*}
	
	\smallskip 
	
	\noindent It is natural to ask the version of these results to the higher dimensions. The question in the higher dimension has no clear answer like the Wermer maximality theorem. The natural object is to generalization of the second result of Wermer, even when considering the algebra generated by polynomials and a pluriharmonic function. For a domain $\Omega\subset \cplx^n,$ let $PH(\Omega)$ denote the class of all pluriharmonic function on $\Omega.$ The works of \v{C}irka \cite{Cirka69}, Izzo \cite{AIJAMS93,AIJ95}, Samuelsson and Wold \cite{SaW12}, and Izzo, Samuelsson, and Wold \cite{ISW16} focused on the study of uniform algebras generated by holomorphic and pluriharmonic functions in higher dimensions.
	In \cite{SaW12}, Samuelsson and Wold \cite{SaW12} proved the following results in the case of the bidisc $\mathbb{D}^2.$
	
	\begin{result}[Samuelsson-Wold]\label{R:SW_T1}
		Let $h_{j}\in PH(\mathbb{D}^{2})\cap\smoo^{1}(\overline{\mathbb{D}}^2)$ for $j=1,\cdots,N.$ Then either there exists a holomorphic disc in $\overline{\mathbb{D}}^2$ where all $h_{j}$'s are holomorphic, or $[z_1,z_2,h_1,\cdots,h_{N};\overline{\mathbb{D}}^2]=\smoo(\overline{\mathbb{D}}^2).$ 	
	\end{result}
	\begin{comment}	
	It is a known fact that if $f \in \smoo(\mathbb{T}^2)$ and if $f$ has a pluriharmonic extension in the bidisc $\mathbb{D}^2$, then $f$ is harmonic in every analytic disc $\phi:\mathbb{D}\to \partial \mathbb{D}^2.$ It is evident that either $\phi(\mathbb{D}) \subset \partial \mathbb{D} \times \mathbb{D}$ or $\phi(\mathbb{D}) \subset \mathbb{D} \times \partial \mathbb{D}.$
	\end{comment}
	The following result can be thought of an analogue of the Wermer maximality theorem in case of the bidisc.
	\begin{result}[Samuelsson-Wold]\label{R:SW_T2}
		Let $f_{j} \in \smoo(\mathbb{T}^2)$ for $j=1,\cdots,N$ with $N\ge 1$, and assume that each $f_{j}$ extends to a pluriharmonic function on $\mathbb{D}^{2}$. Then either $[z_1,z_2,f_1,\cdots,f_{N};\mathbb{T}^2]=\smoo(\mathbb{T}^2)$,  or there exists a non-trivial algebraic variety $Z\subset \cplx^2$ with $\overline{V}\setminus V\subset \mathbb{T}^2,$ and the pluriharmonic extensions of the $f_{j}$'s are holomorphic on $Z,$  where $V=Z\cap (\overline{\mathbb{D}^2}\setminus \mathbb{T}^2).$
	\end{result}

	\begin{remark}
		In \Cref{R:SW_T2} if not all of the functions $f_1,\dots, f_N$ is holomorphic in any analytic disc that lies in $\bdy\mathbb{D}^2$ and $[z_1,z_2,f_1,\cdots,f_{N};\mathbb{T}^2]\ne \smoo(\mathbb{T}^2)$, then the algebraic variety that exists is a distinguished variety.
		As mentioned earlier, by a result of Agler and McCarthy \cite{AglerMcCarthy2005}, every distinguished variety in the bidisc is of the form 
		$\{(z, w) \in \mathbb{D}^2 : \det(\Psi(z) - wI) = 0\}$ for some matrix-valued holomorphic function $\Psi$ on $\mathbb{D}^2$. Therefore, the variety that exists in \Cref{R:SW_T2} is also of the above mentioned determinant form. We do not know what connections are there with the matrix-valued funtion $\Psi$ in \cite{AglerMcCarthy2005} and the pluriharmonic functions in \Cref{R:SW_T2}.
	\end{remark}
	
	\begin{remark}
		It might occur that the variety in \Cref{R:SW_T2} appears in the boundary of the bidisc. In this case, the variety is not a distinguished variety, but such variety can also be explained from the operator theoretic point of view from a result due to Das and Sarkar \cite[Theorem~4.3]{BataJaydeb2017}. From the proof of \Cref{R:SW_T2} it is clear that the form of such variety is $\{\lambda\}\times \mathbb{D}$ or $\mathbb{D}\times\{\lambda\}$ for some $\lambda\in\bdy\mathbb{D}$, which matches with the description in \cite[Theorem~4.3]{BataJaydeb2017}.
	\end{remark}

	Consider the domain $\Omega=\phi(\mathbb{D}^2)$ in $\cplx^2$ and we note that the distinguished boundary of $\Omega$  for the algebra $\mathcal{A}(\Omega)$ is $\Gamma_{\Omega}=\phi(\mathbb{T}^2).$ We prove the following generalization of \Cref{R:SW_T1} and \Cref{R:SW_T2} for the above domain.
	
	\begin{theorem}\label{T:Arbitry_Img_Bidsc}
		Let $h_{j}\in PH(\Omega)\cap\smoo^{1}(\overline{\Omega})$ for $j=1,\cdots,N,$ and $\phi^{-1}(\overline{\Omega})\subset \overline{\mathbb{D}}^2.$ Then, either there exists a holomorphic disc in $\overline{\Omega}$ where all $h_{j}$'s are holomorphic, or $[z_1,z_2,h_1,\cdots,h_{N};\overline{\Omega}]=\smoo(\overline{\Omega}).$ 
	\end{theorem}

	\begin{theorem}\label{T:Arbitry_Img_Torus}
		Let $f_{j}\in \smoo(\Gamma_{\Omega})$ for $j=1,\cdots,N, N\ge 1$ and assume that each $f_{j}$ extends to a pluriharmonic function on $\Omega.$ If $\phi^{-1}(\Gamma_{\Omega})\subset \mathbb{T}^2$. If $f_{j}$ is not holomorphic on any analytic disc present in the boundary $\partial \Omega$ for at least one $j$, then either
		\begin{align*}[z_1,z_2,f_1,\cdots,f_{N};\Gamma_{\Omega}]=\smoo(\Gamma_{\Omega}),
		\end{align*} 
		\noindent or there exists a distinguished variety $V$ in $\Omega$ such that the pluriharmonic extensions of the $f_{j}$'s are holomorphic on $V.$
	\end{theorem}

	%=========================================================
	As a corollary we can extend \Cref{R:SW_T1} and \Cref{R:SW_T2} to the symmetrized bidisc. Recall that the symmetrized bidisc $\mathbb{G}_{2}$ is the image of the bidisc under the  \textit{symmetrization map} $\Pi:(z_1,z_2)\to (z_1+z_2,z_1z_2)$ i.e.,
	\begin{align*}
		\mathbb{G}_{2}=\{(z_1+z_2,z_1z_2):|z_1|<1,|z_2|< 1\}.
	\end{align*}Since $\Pi^{-1}(\Pi(\overline{\mathbb{D}}^2))=\Pi^{-1}(\overline{\mathbb{G}}_2)=\overline{\mathbb{D}}^2,$ by using \Cref{R:Sto_propr}, we get that $\overline{\mathbb{G}}_2$ is polynomially convex. If $f:\mathbb{G}_2\to \cplx$ is a holomorphic function on $\mathbb{G}_2,$ then $f\circ\Pi:\mathbb{D}^2\to\cplx$ is a symmetric function on $\mathbb{D}^{2}.$ Therefore, if $\mathcal{A}(\overline{\mathbb{G}}_2)$ is the algebra of functions that are holomorphic on $\mathbb{G}_2$ and continuous on $\overline{\mathbb{G}}_2,$ then the distinguished boundary $\Gamma_{{\mathbb{G}}_{2}}$ of $\mathbb{G}_2$ is the image $\Pi(\mathbb{T}^2)$ of the torus $\mathbb{T}^2$ (the distinguished boundary of $\mathbb{D}^2$). Since $\mathbb{G}_2$ is neither convex (not even biholomorphic to any convex domain \cite{Costa2004}) nor smooth (not even the Lipschitz domain \cite{DebGorai2015}),  and hence, many results in the theory of several complex variables does not apply to $\mathbb{G}_2.$ Several authors have studied this domain over the last three decades, and it has shown to be a domain with a highly rich complex geometry and function theory: see, among many other articles, \cite{Trybula2015,KosiZwo2016,PflZwo2005,JarPfl2004,EdiZwo2005,Costa2004,AglrYung2004,AglrYngLyk2019,AYL2019,TirPalRoy2012,Sarkar2015}. 
 
	There are significant similarities and contrasts between its geometry and function theory and those of the bidisc. Here we observe that \Cref{R:SW_T1} and \Cref{R:SW_T2} continues to hold if the bidisc is replaced by the symmetrized bidisc. More precisely:

	\begin{corollary}\label{T:Symtric 1}
		Let $h_{j}\in PH(\mathbb{G}_2)\cap\smoo^{1}(\overline{\mathbb{G}}_2)$ for $j=1,\cdots,N.$ Then either there exists a holomorphic disc in $\overline{\mathbb{G}}_2$ where all $h_{j}$'s are holomorphic, or 
		\[
		[z_1,z_2,h_1,\cdots,h_{N};\overline{\mathbb{G}}_2]=\smoo(\overline{\mathbb{G}}_2).
		\]
	\end{corollary}

	\begin{corollary}\label{T:Symtric 2}
		Let $f_{j}\in \smoo(\Gamma_{\mathbb{G}_2})$ for $j=1,\cdots,N, N\ge 1$ and assume that each $f_{j}$ extends to a pluriharmonic function on $\mathbb{G}_2.$  If $f_{j}$ is not holomorphic on any analytic disc present in the boundary $\partial \mathbb{G}_{2}$ for at least one $j$, then either 
		\begin{align*}
			[z_1,z_2,f_1,\cdots,f_{N};\Gamma_{\mathbb{G}_2}]=\smoo(\Gamma_{\mathbb{G}_2}),
		\end{align*} 
		\noindent or there exists a distinguished variety $V$ in $\mathbb{G}_2$ such that the pluriharmonic extensions of the $f_{j}$'s are holomorphic on $V.$
	\end{corollary}
	\begin{remark}
		In view of a result by Pal and Shalit \cite{PalShalit2014}, we see that the variety that appears in \Cref{T:Symtric 2} has the form of a zero set of a certain determinant. However, we do not know whether a similar type of determinant form can also given for the distinguished varieties that appear in \Cref{T:Arbitry_Img_Torus}.
	\end{remark}

	\section{Technical Results}\label{S:technical}

	In this section, we provide some known results and some preliminary lemmas that will be utilized to prove our results.

	\begin{result}[\cite{Sto07}]\label{R:Sto_propr}
		If $F:\mathbb{C}^n\to \mathbb{C}^n$ is a proper holomorphic map, and if $K\subset \mathbb{C}^n$ is a compact set, then the set $K$ is polynomially convex if and only if the set $F^{-1}(K)$ is
		polynomially convex, and $\mathcal{P}(K) =\mathcal{C}(K)$ if and only if $\mathcal{P}(F^{-1}(K)) =\mathcal{C}(F^{-1}(K)).$
	\end{result}

	\begin{result}[Remmert
		Proper Mapping theorem \cite{Rem56,Rem57}]\label{R:Remmert}
		Let $M, N$ be complex spaces, and $f:M\to N$ is a proper holomorphic map. If $Z$ is an analytic subvariety in $M$ then $f(Z)$ is also an analytic subvariety in $N.$ Moreover, if $Z$ is irreducible then $f(Z)$ is also irreducible subvariety of $N.$
	\end{result}
	
	The following result is from the book \cite[Page 29]{Chirka_book89}.
	
	\begin{result}(Chirka)\label{R:Image_AlgVariety}
		Let $\Omega_1\subset \cplx^p,\Omega_2\subset \cplx^m,$ are open subsets such that $\Omega=\Omega_1\times\Omega_2,$ $p+m=n,$ and $\textit{proj}_{1}:(z,w)\to z.$ Let $V$ be an analytic subset in $\Omega$ such that $\textit{proj}_{1}:V\to \Omega_1$ is a proper map. Then $\textit{proj}_{1}(V)$ is an analytic subset in $\Omega_1.$ Moreover, if $\Omega=\cplx^n,$ $\Omega_1=\cplx^p,$ and $V$ is an algebraic subset in $\cplx^n,$ then $\textit{proj}_{1}(V)$ is also an algebraic subset in $\cplx^p.$
	\end{result}
	
	The following lemma is well-known to experts. Since we have not found a explicit mention of this lemma in the literature, we decided to put it here for completeness.	
	\begin{lemma}\label{L:Img_Analytic Variety}
		Let $\Psi:\cplx^n\to \cplx^n$ be a proper polynomial map. Let $Z$ be an algebraic variety in $\cplx^n,$ then $\Psi(Z)$ is also an algebraic variety in $\cplx^n.$ 	
	\end{lemma}
	\begin{proof}
		Consider the algebraic variety $V=\{(\Psi(z),z):z\in Z\}$ in $\cplx^n\times \cplx^n$ and $\Omega_1=\Omega_2=\cplx^n.$ We now show that that $\textit{proj}_{1}:V\to \Omega_1$ is a proper map. Let $K\subset \cplx^n$ be a compact subset of $\cplx^n.$ Then $\textit{proj}_{1}^{-1}\{K\}=(K\times \cplx^n)\cap V=\{(\xi,\eta)\in K\times\cplx^n: (\xi,\eta)\in V\}=\{(\Psi(\eta),\eta)\in K\times\cplx^n:\eta\in Z\}=$compact (since $\Psi$ is a proper map). Therefore, $\textit{proj}_{1}:V\to \Omega_1$ is a proper map. Hence, by \Cref{R:Image_AlgVariety}, we conclude that $\textit{proj}_{1}(V)=\Psi(Z)$ is an algebraic variety.
	\end{proof}

	\begin{remark}
		The case $\Psi=\Pi$ is available in \cite[Lemma 3.1]{PalShalit2014}.	
	\end{remark}	
	%========================
	Let $\Psi:\cplx^n\to \cplx^n$ be a proper holomorphic polynomial map. Let $\Omega:=\Psi(\mathbb{D}^n)$ be a domain such that $\Psi^{-1}(\Psi(\mathbb{D}^n))\subset \mathbb{D}^n,$ $\Psi^{-1}(\Psi(\partial \mathbb{D}^n))\subset \partial\mathbb{D}^n,$ and $\Psi^{-1}(\Psi(\mathbb{T}^n))\subset \mathbb{T}^n.$ The following lemma illustrates that every distinguished variety in $\Omega$ can be derived from a distinguished variety in $\mathbb{D}^n$.

	\begin{lemma}\label{L:Img_Distinguished Variety}
		Let $Z\subset \Omega$. Then $Z$ is a distinguished variety in $\Omega$ if and only if there is
		a distinguished variety $V$ in $\mathbb{D}^n$ such that $\Psi(V)=Z.$
	\end{lemma}
	\begin{proof}
		Given that $\Psi$ is a proper map, it implies that $\Psi$ is onto, and therefore, $\Psi(\Psi^{-1}(Z))=Z$. Additionally, it can be easily demonstrated that $\Psi^{-1}(Z)$ is an algebraic variety. Let us define $V:=\Psi^{-1}(Z)$. Now, we need to prove the following: $V\cap \partial \mathbb{D}^n\subset V\cap\mathbb{T}^n$.
		
		\noindent Consider an element $\alpha\in V\cap \partial \mathbb{D}^n.$ This implies that $\alpha\in \Psi^{-1}(Z)\cap \partial\mathbb{D}^n.$ Hence, we have $\Psi(\alpha)\in Z\cap \Psi(\partial \mathbb{D}^n)$. Since $Z$ is a distinguished variety, we can conclude that $\Psi(\alpha)\in Z\cap \Psi(\partial \mathbb{T}^n)$. Consequently, we can deduce that $\alpha$ lies in $\Psi^{-1}(Z\cap \Psi(\mathbb{T}^n))=\Psi^{-1}(Z)\cap \Psi^{-1}(\Psi(\mathbb{T}^n))$. By our assumption, together with this, we get that $V\cap \partial \mathbb{D}^n\subset V\cap\mathbb{T}^n$.
		
		Conversely, let us assume that $V$ is a subset of $\mathbb{D}^n$ and is a distinguished variety. By using \Cref{L:Img_Analytic Variety}, we can conclude that $\Psi(V)$ is an algebraic variety in $\Omega$. Now, we claim that $Z=\Psi(V)$ is a distinguished variety in $\Omega$. Suppose $\alpha\in Z\cap \Psi(\partial \mathbb{D}^n)=\Psi(V)\cap \Psi(\partial \mathbb{D}^n).$ We need to show that $\alpha$ also lies in $\Psi(\mathbb{T}^n)$. Since $\alpha\in Z\cap \Psi(\partial \mathbb{D}^n)$, there exist $\eta_1\in V$ and $\eta_2\in \partial \mathbb{D}^n$ such that $\Psi(\eta_1)=\Psi(\eta_2)=\alpha$. Consequently, $\eta_2$ belongs to $\Psi^{-1}(\Psi(\partial\mathbb{D}^n))$, which is a subset of $\partial\mathbb{D}^n$. Thus, we have $\eta_2\in V\cap \partial\mathbb{D}^n$, and as a result, $\Psi(\eta_2)\in \Psi(V\cap \partial\mathbb{D}^n)$. This implies that $\alpha$ lies in $\Psi(V\cap \mathbb{T}^n)$.
	\end{proof}

	\begin{remark}
		The case $\Omega=\mathbb{G}_2$ is available in \cite[Lemma 3.1]{PalShalit2014}.   
	\end{remark}

	%========================================
	\begin{lemma}\label{L:Holcomp}
		Let $g:G\subset\cplx^{N}\to\cplx^{N}$ be a proper holomorphic mapping and $q:g(G)\to\cplx$ be a continuous function. If $q\circ g:G\to \cplx$ is holomorphic, then $q$ is holomorphic.
	\end{lemma}
	\begin{proof}
		Let us define $\Omega:=g(G).$ Since $g$ is proper holomorphic, $\Omega$ is open.
		First, we assume $z\in G$ and $\det d{g}(z)\ne 0,$ where $\det dg(z)$ is the determinant of the complex Jacobian matrix of $g$ at $z.$ Then there exists a neighborhood $V$ of $z$ and a neighborhood $W$ of $g(z)$ such that $g^{-1}:W\to V$ is holomorphic. Therefore, $q\circ g\circ g^{-1}=q$ is holomorphic at $g(z).$ Next, we define $X:=\{z\in G:\det d{g}(z)=0\}.$ Hence, $q$ is holomorphic on $\Omega\setminus g(X).$ Clearly, $X$ is an analytic variety with $\dim_{\cplx}X\le (N-1).$ Since $g$ is proper holomorphic mapping, by \Cref{R:Remmert}, $g(X)$ is also an analytic variety in $\Omega.$ Since $q$ is continuous on $\Omega$ and holomorphic on $\Omega\setminus g(X),$ by Riemann's removable singularity theorem, we can say that $q$ is holomorphic on $\Omega.$
	\end{proof}

	Let $\Psi:\cplx^n\to \cplx^n$ be a proper holomorphic map. Let $\Omega:=\Psi(\mathbb{D}^n)$ be a domain such that $\Psi^{-1}(\Psi(\mathbb{D}^n))\subset \mathbb{D}^n,$ $\Psi^{-1}(\Psi(\partial \mathbb{D}^n))\subset \partial\mathbb{D}^n,$ and $\Psi^{-1}(\Psi(\mathbb{T}^n))\subset \mathbb{T}^n.$ We denote the distinguished boundary of $\Omega$ for the algebra $\mathcal{A}(\Omega)$ by $\Gamma_{\Omega}.$ Clearly, $\Gamma_{\Omega}$ is equal to $\Psi(\mathbb{T}^n).$ 
	
	\smallskip
	
	The following theorem might be of independent interest. We will use this in our proofs.
	
	\begin{theorem}\label{L: Approx_Cont_Func_SymmBidsc}
		Let $N\ge 1$ and $f_1,\cdots,f_{N}\in \smoo(\Gamma_{\Omega}).$ Then $[z_1,\cdots,z_n,f_{1},\cdots,f_N;\Gamma_{\Omega}]=\smoo(\Gamma_{\Omega})$ if and only if $\Gr_{f}(\Gamma_{\Omega})$ is polynomially convex, where $f=(f_1,\cdots,f_{N}).$  
	\end{theorem}	
	
	\begin{proof}
		We denote $X:=\Gr_{f}(\Gamma_{\Omega}).$ Since $[z_1,\cdots,z_n,f_{1},\cdots,f_N;\Gamma_{\Omega}]=\smoo(\Gamma_{\Omega})$ implies $\poly(X)=\smoo(X),$ hence $\widehat{X}=X.$
		
		\noindent Conversely, suppose that $\widehat{X}=X.$ We consider the proper holomorphic map $\Phi:\cplx_{z}^{n}\times\cplx_{w}^{N}\to \cplx_{z}^{n}\times\cplx_{w}^{N},$ define by
		\begin{align*}
			\Phi(z,w)=(\Psi(z),w). 	
		\end{align*}
		Clearly,
		\begin{align*}
			\Phi^{-1}(X)=\Gr_{f\circ \Psi}(\mathbb{T}^{n})=:Y.
		\end{align*}
		Since $X$ is polynomially convex, $Y$ is also polynomially convex (by \Cref{R:Sto_propr}). Let $U$ be a neighborhood of $\mathbb{T}^{n}$ such that $z_{1}\not =0~~\text{ on } U.$ Define $g(z_1,z_2,\cdots,z_n)=\frac{1}{z_1}.$ Then $g$ is holomorphic on $U.$ Also, $g$ is holomorphic on $U\times \mathbb{C}^{N}.$ Since $Y\subset U\times\cplx^{N},$ by the \textit{Oka-Weil} approximation theorem, there exists a sequence of polynomial $P_{j}$ in $\cplx^{n}_{z}\times\cplx_{w}^{N}$ such that
		$P_{j}(z,w)\to g$  uniformly on $Y.$ This implies $P_{j}(z,(f\circ \Psi)(z))\to g=\frac{1}{z_{1}}=\overline{z}_{1}$ uniformly on  $\mathbb{T}^{n}.$ Hence $\overline{z}_{1}\in [z_1,\cdots,z_n,f_{1}\circ\Psi,\cdots,f_N\circ\Psi;\mathbb{T}^{n}].$ By the similar method we can show that $\overline{z}_j\in [z_1,\cdots,z_n,f_{1}\circ\Psi,\cdots,f_N\circ\Psi;\mathbb{T}^{n}],~\forall_{j}\in\{1,\cdots,n\}.$ \sloppy Hence, $[z_1,\cdots,z_n,\overline{z}_{1},\cdots,\overline{z}_n;\mathbb{T}^n]\subset [z_1,\cdots,z_n,f_{1}\circ\Psi,\cdots,f_N\circ\Psi;\mathbb{T}^{n}].$ Therefore,
		\begin{align}\label{E:Approx_Torus}
			[z_1,\cdots,z_n,\overline{z}_{1},\cdots,\overline{z}_n;\mathbb{T}^{n}]=\smoo(\mathbb{T}^{n}) =[z_1,\cdots,z_n,f_{1}\circ\Psi,\cdots,f_N\circ\Psi;\mathbb{T}^{n}].
		\end{align}
		Note that $\poly(X)=\smoo(X)$ if and only if $\poly(\Phi^{-1}(X))=\smoo(\Phi^{-1}(X))$ (see \Cref{R:Sto_propr}) i.e., $\poly(Y)=\smoo(Y).$ Therefore, using (\ref{E:Approx_Torus}), we get that
		\begin{align*}
			[z_1,\cdots,z_n,f_{1},\cdots,f_N;\Gamma_{\Omega}]=\smoo(\Gamma_{\Omega}).	
		\end{align*}
	\end{proof}
	
	\begin{corollary}\label{C: Approx_Cont_Func_SymmBidsc}
		Let $N\ge 1,$ and $f_1,\cdots,f_{N}\in \smoo(\Gamma_{\mathbb{G}_{n}}).$ Then $[z_1,\cdots,z_n,f_{1},\cdots,f_N;\Gamma_{\mathbb{G}_{n}}]=\smoo(\Gamma_{\mathbb{G}_{n}})$ if and only if $\Gr_{f}(\Gamma_{\mathbb{G}_{n}})$ is polynomially convex, where $f=(f_1,\cdots,f_{N}).$  
	\end{corollary}

	In \cite{Jimbo03, Jimbo05}, Jimbo explored the structure of polynomial hulls concerning graphs of antiholomorphic polynomials on the torus. For the sake of completeness, we include Jimbo's result from \cite{Jimbo05} here since we will use it multiple times in this paper.
	Let $\mathbb{T}^2$ be the torus in $\cplx^2$ and $P$ be an arbitrary polynomial in $\cplx^{2}.$ In \cite{Jimbo05}, Jimbo gave a description for $\widehat{\Gr_{\overline{P}}(\mathbb{T}^2)}.$ Let the polynomial $P(z_1,z_2)$ be of degree $m$ in $z_1$ and of degree $n$ in $z_2.$ We write
	\begin{align*}
		P(z_1,z_2))=	\sum_{\substack{0\le i\le m\\
				0\le j\le n	\\
		}}
		a_{ij}z_1^{i}z_2^{j}.
	\end{align*}
	
	\noindent Therefore, on $\mathbb{T}^2,$ we have
	\begin{align*}
		\overline{P(z_1,z_2)}&
		=\frac{1}{z^{m}_{1}z^{n}_{2}}	\sum_{\substack{0\le i\le m\\
				0\le j\le n	\\
		}}
		\overline{a}_{ij}z_{1}^{m-i}{z_2}^{n-j}\\
		&=\frac{K(z_1,z_2)}{z^{m}_{1}z^{n}_{2}}=h(z_1,z_2),\text{ where } K(z_1,z_2)=\sum_{\substack{0\le i\le m\\
				0\le j\le n	\\
		}}
		\overline{a}_{ij}z_1^{m-i}z_2^{n-j}.
	\end{align*}
	\noindent Hence on $\mathbb{T}^2,$ we get that
	\begin{align*}
		\overline{P(z_1,z_2)}=h(z_1,z_2),\text{ where } h(z_1,z_2)=\frac{K(z_1,z_2)}{z_1^m z_2^n}.
	\end{align*}
	\noindent We define $L:=\{z_1=0,|z_2|\le 1\}\cup\{z_2=0,|z_1|\le 1\}$ and
	\begin{align}\label{E:hull_interior}
		X=\left\{(z_1,z_2)\in \overline{\mathbb{D}}^{2}\setminus (L\cup \mathbb{T}^2): \overline{P(z_1,z_2)}=h(z_1,z_2)\right\}.
	\end{align}

	\noindent We set
	\begin{align*}
		\triangle({z}):=
		\begin{vmatrix} 
			\frac{\partial P(z)}{\partial{z_1}}& \frac{\partial P(z)}{\partial z_2}\\[1.5ex]
			\frac{\partial h(z)}{\partial{{z_1}}}& \frac{\partial h(z)}{\partial{z_2}}\\[1.5ex]
		\end{vmatrix}.
	\end{align*}	
	
	\noindent We can write 
	\begin{align*}
		\triangle(z)
		=\frac{1}{z^{m+1}_{1}z^{n+1}_{2}}\prod^{l}_{j=1}q_{j}(z),
	\end{align*}
	where each $q_{j}$ is an irreducible polynomial in $\cplx^2.$ We define the corresponding irreducible algebraic variety $Z_{j}:=Z(q_{j})=\{z\in \cplx^2:q_{j}(z)=0\}.$ We assume $\triangle({z})\not\equiv 0$ on $X.$ Therefore, each $q_{j}$ is a non-zero holomorphic polynomial in $\cplx^2.$ 
	
	\smallskip
	
	\noindent We denote $Q_{j}=Z_{j}\cap \mathbb{T}^2.$
	
	\begin{result}[Jimbo]\label{R:Descriptn_Hull_Jimbo}
		We let $J=\{j\in \{1,\cdots,l\}: \emptyset\ne Q_{j}\ne \widehat{Q_{j}}, \widehat{Q_{j}}\setminus L\subset X\}.$ 
		\begin{enumerate}
			\item [(i)]If $J=\emptyset,$ then $\widehat{\Gr_{\overline{P}}(\mathbb{T}^2)}=\Gr_{\overline{P}}(\mathbb{T}^2),$ and $	[z_1,z_2,\overline{P};\mathbb{T}^2]=\smoo(\mathbb{T}^2);$
			\item[(ii)] If $J\ne \emptyset,$ then 
		\end{enumerate}	
		\begin{align*}
			\widehat{\Gr_{\overline{P}}( \mathbb{T}^2)}=\Gr_{\overline{P}}(\mathbb{T}^2)\cup\bigg(\cup_{j\in J}  \Gr_{\overline{P}}(\widehat{Q_{j}})\bigg).
		\end{align*}
	\end{result}

	\section{proof of \Cref{T:Arbitry_Img_Bidsc,T:Arbitry_Img_Torus}}
	Note that the map $\phi:\cplx^2\to \cplx^2$ is defined as $\phi(z)=(p_{1}(z),p_{2}(z)).$ We consider the proper holomorphic map $\widetilde{\Psi}:\cplx^{2+N}\to \cplx^{2+N},$ defined as follows:
	\begin{align}\label{E:GenProperMap}
		\widetilde{\Psi}(z_1,z_2,w_1,\cdots,w_{N})=\left(\phi(z_1,z_2),w_1,\cdots,w_N\right),
	\end{align}
	where $~ (z_1,z_2)\in\cplx^2,$ and $(w_1,\cdots,w_N)\in\cplx^N.$
	Recall that $\Omega=\phi(\mathbb{D}^2)$ and $\Gamma_{\Omega}=\phi(\mathbb{T}^2).$\\
	
	\begin{proof}[	Proof of \Cref{T:Arbitry_Img_Bidsc}]

		We claim that $\widetilde{\Psi}^{-1}(\Gr_{h}(\overline{\Omega}))=\Gr_{h\circ\phi}(\mathbb{\overline{D}}^2)$: let 
		\begin{align*}
			(\alpha,\beta)\in \widetilde{\Psi}^{-1}(\Gr_{h}(\overline{\Omega}))&
			\implies  \widetilde{\Psi}(\alpha,\beta)\in \Gr_{h}(\overline{\Omega})\\
			& \implies (\phi(\alpha),\beta)\in \Gr_{h}(\overline{\Omega})\\
			& \implies \beta=h(\phi(\alpha)) \text{ and } \phi(\alpha)\in \overline{\Omega}. 
		\end{align*}
		Now 
		\begin{align*}
			\phi(\alpha)\in \overline{\Omega} & \implies \alpha
			\in \phi^{-1}(\phi(\alpha))\subset \phi^{-1}( \overline{\Omega})\subset \overline{\mathbb{D}}^2.
		\end{align*}
		Therefore
		$\widetilde{\Psi}^{-1}(\Gr_{h}(\overline{\Omega}))\subset \Gr_{h\circ\phi}(\mathbb{\overline{D}}^2).$ \\
		\vspace{2mm} 
		\noindent Conversely, let
		\begin{align*}
			(p,q)\in \Gr_{h\circ \phi}(\overline{\mathbb{D}}^2) &\implies 
			q=(h\circ \phi)(p) \text{ and } p\in \overline{\mathbb{D}}^2\\
			&\implies q=h(\phi(p)) \text{ and } \Pi(p)\in \overline{\Omega}\\
			&\implies (\phi(p),q)\in \Gr_{h}(\overline{\Omega})\\
			&\implies \widetilde{\Psi}(p,q)\in \Gr_{h}(\overline{\Omega})\\
			&\implies (p,q)\in \widetilde{\Psi}^{-1}\left(\Gr_{h}(\overline{\Omega})\right).
		\end{align*}
		Hence $\Gr_{h\circ\phi}(\mathbb{\overline{D}}^2)\subset\widetilde{\Psi}^{-1}(\Gr_{h}(\overline{\Omega})).$ 
		Therefore,
		$\widetilde{\Psi}^{-1}(\Gr_{h}(\overline{\Omega}))=\Gr_{h\circ\phi}(\mathbb{\overline{D}}^2).$	
		Since $\widetilde{\Psi}$ is proper holomorphic mapping and $\widetilde{\Psi}^{-1}(\Gr_{h}(\overline{\Omega}))=G_{h\circ\phi}(\mathbb{\overline{D}}^2),$ by \Cref{R:Sto_propr}, we can say that  $\poly\left(\Gr_{h}(\overline{\Omega})\right)=\smoo\left(\Gr_{h}(\overline{\Omega})\right)$ if and only if $\poly\left(\Gr_{h\circ\phi}(\overline{\mathbb{D}}^2)\right)=\smoo\left(\Gr_{h\circ\phi}(\overline{\mathbb{D}}^2)\right).$ We note that $h\circ \phi$ is pluriharmonic on $\mathbb{D}^2$ and continuous on $\overline{\mathbb{D}}^2.$ Therefore, two cases hold.
		
		\vspace{1.5mm}
		
		\noindent  {\bf Case I:} $\poly\left(\Gr_{h\circ\phi}(\overline{\mathbb{D}}^2)\right)=\smoo\left(\Gr_{h\circ\phi}(\overline{\mathbb{D}}^2)\right).$ In this case we have
		$\poly\left(\Gr_{h}(\overline{\Omega})\right)=\smoo\left(\Gr_{h}(\overline{\Omega})\right).$\\
		
		\noindent  {\bf Case II:} $\poly\left(\Gr_{h\circ\phi}(\overline{\mathbb{D}}^2)\right)\ne\smoo\left(\Gr_{h\circ\phi}(\overline{\mathbb{D}}^2)\right).$
		Therefore, by \Cref{R:SW_T1}, there exists an analytic disc $g:\mathbb{D}\hookrightarrow{} \overline{\mathbb{D}}^2$ where $(h_{j}\circ\phi)\circ g:\mathbb{D}\hookrightarrow{} \mathbb{\overline{D}}^2$ is holomorphic for all $j=1,\cdots,N.$ If we take $\gamma:=\phi\circ g,$ then clearly $\gamma:\mathbb{D}\hookrightarrow{}\overline{\Omega}$ is an analytic disc in $\overline{\Omega}$ such that $h_j$ is holomorphic on $\gamma(\mathbb{D})$ (by \Cref{L:Holcomp}) for all $j=1,\cdots,N.$ This proves the theorem.
	\end{proof}	
	
	\begin{proof}[Proof of \Cref{T:Arbitry_Img_Torus}]
		Let $h_{j}$ denotes the pluriharmonic extension of $f_{j}$ to $\Omega$ and write $h=(h_1,\cdots,h_{N}):\overline{\Omega}\to \cplx^{N}.$ We have $\widetilde{\Psi}$ is proper holomorphic mapping and $\widetilde{\Psi}^{-1}(\Gr_{h}(\Gamma_{\Omega}))=\Gr_{h\circ\phi}(\mathbb{T}^2).$ Therefore, by \Cref{R:Sto_propr}, $\Gr_{h}(\Gamma_{\Omega})$ is polynomially convex if and only if $\Gr_{h\circ\phi}(\mathbb{T}^2)$ is polynomially convex. We note that $h\circ \phi$ is pluriharmonic on $\mathbb{D}^2$ and continuous on $\overline{\mathbb{D}}^2.$ Therefore, two cases hold.
		
		\vspace{1mm}
		
		\noindent  {\bf Case I:} $\Gr_{h}(\Gamma_{\Omega})$ is polynomially convex. In view of \Cref{L: Approx_Cont_Func_SymmBidsc}, we have $$[z_1,z_2,f_{1},\cdots,f_N;\Gamma_{\Omega}]=\smoo(\Gamma_{\Omega}).$$
		
		\noindent  {\bf Case II:} $\Gr_{h}(\Gamma_{\Omega})$ is not polynomially convex. Consequently, $\Gr_{h\circ\phi}(\mathbb{T}^2)$ is not polynomially convex. Therefore, by \Cref{R:SW_T2}, there exists a distinguished variety $Z\subset\mathbb{D}^2$ where $(h_{j}\circ\phi)$ is holomorphic for all $j=1,\cdots,N.$ Since $\phi$ is a proper holomorphic mapping, by \Cref{L:Img_Analytic Variety}, we have $\phi(Z)$ is also an algebraic variety. Since $\phi$ is proper holomorphic, $(h_{j}\circ\phi)$ is holomorphic on $Z,$ then $h_{j}$ is also holomorphic on $\phi(Z)$ (by \Cref{L:Holcomp}). Since $\phi$ sends distinguished variety of $\mathbb{D}^2$ to distinguished variety of $\Omega$ (\Cref{L:Img_Distinguished Variety}), we have $\phi(Z)\cap b\Omega\subset \Gamma_{\Omega}.$ 
	\end{proof}

	\section{Description of Polynomial Hull }
	In this section, we provide a description of the polynomial convex hull of the graph of an anti-holomorphic polynomial over the distinguished boundary of the domain $\Omega,$ where $\Omega$ is the image of the bidisc under certain proper polynomial map from $\cplx^2$ to $\cplx^2.$

	\par Let $F=(f_1,f_2,\cdots,f_{n}):\cplx^n\to \cplx^n$ be a proper map. Let 	$$J_{f}(z)=
	\begin{vmatrix} 
		\frac{\partial f_{1}}{\partial{z_1}}(z) &  \frac{\partial f_{1}}{\partial{z_2}}(z)&\cdots&\frac{\partial f_{1}}{\partial{z_n}}(z)\\[1.5ex]
		\vdots&\vdots&\cdots&\vdots\\[1.5ex]
		\frac{\partial f_{n}}{\partial{z_1}}(z) & \frac{\partial f_{n}}{\partial{z_2}}(z)&\cdots&\frac{\partial f_{n}}{\partial{z_n}}(z)\\[1.5ex]
	\end{vmatrix}.$$
	The \textit{critical locus} of $f$ is the complex analytic variety $Z(J_{f})=\{z\in \cplx^n:J_{f}(z)=0\}\subset \cplx^n.$ The \textit{branch locus} $B(f)$ of $f$ is the image of the critical locus. Since $f$ is proper,
	\begin{align*}
		f:\cplx^n\setminus f^{-1}(B(f))\to \cplx^n\setminus B(f)
	\end{align*}
	is a covering map of finite degree $d;$  $d$ is said to be the \textit{topological degree} of $f.$

	\begin{definition}
		Two proper map $\phi,\tilde{\phi}:\cplx^{2}\to \cplx^2$ are said to be \textit{equivalent} if there exist $f,g\in \text{Aut}(\cplx^2)$ such that $\phi=f\circ\tilde{\phi}\circ g.$ 
	\end{definition}

	%===========================
	
	Consider two holomorphic polynomials, $p_1$ and $p_2$, defined in $\mathbb{C}^2$. Let $\phi(z) = (p_1(z), p_2(z))$ represent a proper holomorphic mapping from $\mathbb{C}^2$ to $\mathbb{C}^2$, equivalent to $\tilde{\phi}(z_1, z_2) = (z_1^m, z_2^n)$ for some natural numbers $m$ and $n$.
	There is a characterization due to Lamy \cite{Lamy05} (see also Bisi and Polizzi \cite{BisiPolizzi2010}) for $m=1$ and $n=2$ as follows: a proper polynomial map $f:\mathbb{C}^2\to \mathbb{C}^2$ with a topological degree of 2 is equivalent to $g(z_1,z_2)=(z_1,z^{2}_2).$

	\medskip
	
	Let $P(z_1, z_2)$ be any polynomial in $\mathbb{C}^2$. We aim to calculate $\widehat{\Gr_{\overline{P}}(\Gamma_{\Omega})}$. It is evident that $\widetilde{\Psi}^{-1}(\Gr_{\overline{P}}(\Gamma_{\Omega})) = \Gr_{\overline{P}\circ\phi}(\mathbb{T}^2) = \Gr_{\overline{P\circ\phi}}(\mathbb{T}^2)$ ($\widetilde{\Psi}$ is given by (\ref{E:GenProperMap})). Consequently, $\Gr_{\overline{P}}(\Gamma_{\Omega}) = \widetilde{\Psi}\left(\operatorname{Gr}_{\overline{P\circ\phi}}(\mathbb{T}^2)\right)$. In this scenario, the following result holds.

	\begin{lemma}\label{L:ImgHull_Graph_Pro_Map_D2}
		$\widehat{\widetilde{\Psi}(Y)}=\widetilde{\Psi}\left(\widehat{Y}\right),$ where $Y=\Gr_{\overline{P\circ\phi}}(\mathbb{{T}}^2).$
	\end{lemma}
	
	\begin{proof}
		Since $\widetilde{\Psi}$ is a proper holomorphic map, by using \Cref{R:Sto_propr}, we have that $\widetilde{\Psi}^{-1}\left(\widehat{\widetilde{\Psi}(Y)}\right)$ is polynomially convex. Therefore
		\begin{align*}
			\widehat{Y}\subset \widehat{\widetilde{\Psi}^{-1}\left(\widehat{\widetilde{\Psi}(Y)}\right)}\subset \widetilde{\Psi}^{-1}\left(\widehat{\widetilde{\Psi}(Y)}\right).
		\end{align*}
		This implies, $\widetilde{\Psi}(\widehat{Y})\subset \widehat{\widetilde{\Psi}(Y)}.$
		
		\noindent Next, we show that $\widetilde{\Psi}^{-1}\left(\widetilde{{\Psi}}(\widehat{Y})\right)\subset \widehat{Y}.$ To prove this, let $(\alpha_1,\alpha_2,\beta)\in \widetilde{\Psi}^{-1}\left(\widetilde{{\Psi}}(\widehat{Y})\right).$ Then there exists $(\xi_1,\xi_2,\eta)\in \widehat{Y}$ such that $\widetilde{\Psi}(\alpha_1,\alpha_2,\beta)=\Psi(\xi_1,\xi_2,\eta).$ This implies, $\phi(\alpha_1,\alpha_2)=\phi(\xi_1,\xi_2)$ and $\beta=\eta.$ Since, $\phi$ is proper polynomial map and is equivalent to $\tilde{\phi}(z_1,z_2)=(z^m_1,z^n_{2}),$ there exist $f,g\in \text{Aut}(\cplx^2)$ such that $\phi=f\circ \tilde{\phi} \circ g.$  Then 
		
		\begin{align*}
			&\phi(\alpha_1,\alpha_2)=\phi(\xi_1,\xi_2)\\
			\implies &(f\circ \tilde{\phi} \circ g)(\alpha_1,\alpha_2)=(f\circ \tilde{\phi} \circ g)(\xi_1,\xi_2)\\
			\implies &(\tilde{\phi} \circ g)(\alpha_1,\alpha_2)=(\tilde{\phi} \circ g)(\xi_1,\xi_2)\\
			\implies& g^{m}_{1}(\alpha_1,\alpha_2)=g^{m}_{1}(\xi_1,\xi_2) \text{ and } g^{n}_2(\alpha_1,\alpha_2)=g^{n}_2(\xi_1,\xi_2), \text{ where } g=(g_1,g_2).\\
			\implies &  (\alpha_1,\alpha_2)=g^{-1}\{(\lambda^{k}_{m}g_1(\xi_1,\xi_2),\lambda^{r}_{n}g_2(\xi_1,\xi_2)\}=(a_{k},b_{r}),
		\end{align*}
		where $\lambda_{l}=\cos \frac{2\pi}{l}+i\sin\frac{2\pi}{l},$ $k\in\{0,\cdots,m-1\}$ and $ r\in \{0,\cdots,n-1\}.$
		
		\smallskip
		
		\noindent It  remains to show that $(a_{k},b_{r},\eta)\in \widehat{Y}.$ If possible, assume that $(a_{k},b_{r},\eta)\notin \widehat{Y}$ for some $k\in \{0,\cdots,m-1\},r\in \{0,\cdots,n-1\}.$ Then there exists a polynomial $\chi$ in $\cplx^2_{z}\times\cplx_{w}$ such that 
		\begin{align}\label{E:Other_pt_InHull}
			|\chi(a_{k},b_{r},\eta)|>\sup_{Y}|\chi(z,w)|.
		\end{align}
		
		\noindent Let us define $F(z_1,z_2):=(\lambda^{k}_{m}z_1,\lambda^{r}_{n}z_{2}),$ and $\tilde{F}(z_1,z_2,w):=((g^{-1}\circ F\circ\ g)(z),w).$  Since $\phi^{-1}(\phi (\mathbb{T}^2))\subset \mathbb{T}^2$ (hence $(g^{-1}\circ F\circ\ g)(z)\in \mathbb{T}^2$ if $z\in \mathbb{T}^2$), using (\ref{E:Other_pt_InHull}), we get that
		\begin{align}\label{E:pt_InHull}
			|(\chi\circ \tilde{F})(\xi,\eta)|>\sup_{Y}|(\chi\circ \tilde{F})(z,w)|.
		\end{align}
		Since $\tilde{F}\in \text{Aut}(\cplx^{3}),$ (\ref{E:pt_InHull}) says that $(\xi,\eta)\notin \widehat{Y}$ and this is a contradiction. Hence $(a_{k},b_{r},\eta)\in \widehat{Y}.$ Therefore, $\widetilde{\Psi}^{-1}\left(\widetilde{\Psi}(\widehat{Y})\right)=\widehat{Y}.$  
		Since $\widetilde{\Psi}$ is proper holomorphic map, by using \Cref{R:Sto_propr}, we can say that $\widetilde{\Psi}(\widehat{Y})$ is polynomially convex. Therefore, $\widehat{\widetilde{\Psi}(Y)}\subset \widetilde{\Psi}(\widehat{Y}).$ This proves the lemma.
	\end{proof}	
	
	By using \Cref{L:ImgHull_Graph}, we can say that
	
	\begin{align*}
		\widehat{\Gr_{\overline{P}}({\Gamma_{\Omega})}}=\widetilde{\Psi}\left(\widehat{\Gr_{\overline{P\circ\phi}}(\mathbb{{T}}^2)}\right).
	\end{align*}
	Therefore, to give  a description for $\widehat{\Gr_{\overline{P}}({\Gamma_{\Omega})}},$ it is enough to compute $\widehat{\Gr_{\overline{P\circ\phi}}(\mathbb{{T}}^2)}.$

	\medskip
	
	\subsection{ Description of Hull on Symmetrized Bidisc} Let $P(z_1,z_2)$ be any polynomial in $\cplx^{2}.$ By \Cref{L:ImgHull_Graph_Pro_Map_D2}, we calculate  $\widehat{\Gr_{\overline{P}}(\Gamma_{{\mathbb{G}}_{2}})}.$ If we take $p_{1}(z)=z_1+z_2$ and $p_{1}(z)=z_1z_2,$ then $\phi=\Pi$ and $\widetilde{\Psi}(z,w)=(\Pi(z),w)$ is a proper map from $\cplx^3$ to $\cplx^3.$ It is easy to show that $\Pi$ a proper polynomial map of topological degree $2,$ and hence equivalent to $(z_1,z^2_{2}).$ Clearly, $\widetilde{\Psi}^{-1}(\Gr_{\overline{P}}({\Gamma}_{\mathbb{G}_2}))=\Gr_{\overline{P}\circ\Pi}(\mathbb{{T}}^2)=\Gr_{\overline{P\circ\Pi}}(\mathbb{{T}}^2).$ Therefore, $\Gr_{\overline{P}}({\Gamma}_{\mathbb{G}_2}))=\widetilde{\Psi}\left(\Gr_{\overline{P\circ\Pi}}(\mathbb{{T}}^2)\right).$
	
	\smallskip
	
	By \Cref{L:ImgHull_Graph_Pro_Map_D2}, we get that
	\begin{lemma}\label{L:ImgHull_Graph}
		$\widehat{\widetilde{\Psi}\left(Y\right)}=\widetilde{\Psi}\left(\widehat{Y}\right),$ where $Y=\Gr_{\overline{P\circ\Pi}}(\mathbb{{T}}^2).$	
\end{lemma}

	By using \Cref{L:ImgHull_Graph}, we can say that
	
	\begin{align*}
		\widehat{\Gr_{\overline{P}}({\Gamma_{\mathbb{G}_2})}}=\widetilde{\Psi}\left(\widehat{\Gr_{\overline{P\circ\Pi}}(\mathbb{{T}}^2)}\right).
	\end{align*}
	Therefore, to give  a description for $\widehat{\Gr_{\overline{P}}({\Gamma_{\mathbb{G}_2})}},$ it is enough to compute $\widehat{\Gr_{\overline{P\circ\Pi}}(\mathbb{{T}}^2)}.$

	\section{Examples}
	\begin{example}
		Let $P(z_1,z_2)=z_1-z_2.$  Then $[z_1,z_2,\overline{P};\Gamma_{\mathbb{G}_{2}}]\ne \smoo(\Gamma_{\mathbb{G}_{2}}).$
	\end{example}	
	
	\noindent {\bf Explanation}
	In view of \Cref{C: Approx_Cont_Func_SymmBidsc}, to demonstrate that $[z_1,z_2,\overline{P};\Gamma_{\mathbb{G}_{2}}]\ne \smoo(\Gamma_{\mathbb{G}_{2}}),$ it suffices to establish that the graph of $\overline{P}$ over $\Gamma_{\mathbb{G}_{2}}$ is not polynomially convex. To achieve this, it is sufficient to show that the graph of $\overline{P\circ \Pi}$ over $\Gamma_{\mathbb{G}_{2}}$ lacks polynomial convexity. Following the notation in \Cref{R:Descriptn_Hull_Jimbo}, we define $h(z)=\frac{1}{z_1}+\frac{1}{z_2}-\frac{1}{z_1z_2}.$ Then
	
	\begin{align*}
		\triangle({z})=
		\begin{vmatrix} 
			\frac{\partial (P\circ\Pi)}{\partial{z_1}}& \frac{\partial (P\circ\Pi)}{\partial z_2}\\[1.5ex]
			\frac{\partial h}{\partial{{z_1}}}& \frac{\partial h}{\partial{z_2}}\\[1.5ex]
		\end{vmatrix}
		=
		\begin{vmatrix} 
			1-z_2& 1-z_1\\[1.5ex]
			\frac{-1}{z^2_1}+\frac{1}{z^2_1z_2}& \frac{-1}{z^2_2}+\frac{1}{z^2_2z_1}\\[1.5ex]
		\end{vmatrix}
		&=\frac{1}{z^2_{1}z^{2}_{2}}(z_1-z_2)(z_1-1)(z_2-1).
	\end{align*}
	\noindent We define $q_{1}:=z_1-1,~~q_2=z_2-1,q_3:=z_{1}-z_{2},$ and $Z_{j}=\{z\in \cplx^{2}:q_{j}(z)=0\},j=1,2,3.$ Therefore, 	
	\begin{align*}
		\Sigma&=\left\{z\in \overline{\mathbb{D}}^{2}\setminus (L\cup \mathbb{T}^2): \triangle(z)=0\right\}\\
		&=\left\{z\in \overline{\mathbb{D}}^{2}\setminus (L\cup \mathbb{T}^2)\right\}\cap[\cup^{3}_{j=1} Z_{j}],
	\end{align*}
	and	
	\begin{align*}
		X&=\left\{z\in \overline{\mathbb{D}}^{2}\setminus (L\cup \mathbb{T}^2): \overline{(P\circ \Pi)(z)}=h(z)\right\}\\
		&=\left\{z\in \overline{\mathbb{D}}^{2}\setminus (L\cup \mathbb{T}^2): \overline{z_1+z_2-z_1z_2}=\frac{1}{z_1}+\frac{1}{z_1}-\frac{1}{z_1z_2}\right\}.
	\end{align*}
	
	\noindent Here $Q_{j}=Z_{j}\cap \mathbb{{T}}^2.$ Clearly,
	\begin{align*}
		\widehat{Q_1}&=\{z\in \cplx^2:z_1=1,|z_2|\le 1\}\ne Q_1;\\
		\widehat{Q_2}&=\{z\in \cplx^2:z_2=1,|z_1|\le 1\}\ne Q_2;\\
		\widehat{Q_3}&=\{z\in \cplx^2:z_1=z_2,|z_1|\le 1\}\ne Q_3.
	\end{align*}
	
	It is evident that $\widehat{Q_{j}}\setminus (\mathbb{T}^2\cup L)\subset X$ holds true only for $j=1,2.$ On the other hand, we note that $(\frac{1}{2},\frac{1}{2})\in \widehat{Q_{3}}\setminus (\mathbb{T}^2\cup L),$ yet $(\frac{1}{2},\frac{1}{2})\notin X.$ Therefore, by \Cref{R:Descriptn_Hull_Jimbo}, we deduce that:
	
	\begin{align*}
		\widehat{\Gr_{\overline{P\circ \Pi}}(\mathbb{T}^2)}=\Gr_{\overline{P\circ \Pi}}(\mathbb{T}^2)\cup \Gr_{\overline{P\circ \Pi}}(\widehat{Q_{1}})\cup \Gr_{\overline{P\circ \Pi}}(\widehat{Q_{2}}).
	\end{align*}
	
	Hence 
	\begin{align*}
		\widehat{\Gr_{\overline{P}}(\Gamma_{\mathbb{G}_{2}})}&=\Psi\left(\Gr_{\overline{P\circ \Pi}}(\mathbb{T}^2) \right) \cup \Psi\left(\Gr_{\overline{P\circ \Pi}}(\widehat{Q_{1})}\right)\cup \Psi\left( \Gr_{\overline{P\circ \Pi}}(\widehat{Q_{2})}\right)\\
		&=	\Gr_{\overline{P}}(\Gamma_{\mathbb{G}_{2}})\cup \{(1+z,z,w):w=\overline{P(1+z,z)}, z\in \overline{\mathbb{D}}\}\\
		&\cup \{(1+z,z,w):w=\overline{P(1+z,z)}, z\in \overline{\mathbb{D}}\}\\
		&=	\Gr_{\overline{P}}(\Gamma_{\mathbb{G}_{2}})\cup \{(1+z,z,w):w=\overline{P(1+z,z)}, z\in \overline{\mathbb{D}}\}\\
		&=	\Gr_{\overline{P}}(\Gamma_{\mathbb{G}_{2}})\cup \{(1+z,z,1):z\in \overline{\mathbb{D}}\}.
	\end{align*}
	
	\begin{example}
		$P(z_1,z_2)=z_1-2z_{2}.$  Then $[z_1,z_2,\overline{P};\Gamma_{\mathbb{G}_{2}}]= \smoo(\Gamma_{\mathbb{G}_{2}}).$
	\end{example}	
	
	\noindent {\bf Explanation} In light of \Cref{C: Approx_Cont_Func_SymmBidsc}, in order to establish that $[z_1,z_2,\overline{P};\Gamma_{\mathbb{G}_{2}}]= \smoo(\Gamma_{\mathbb{G}_{2}}),$ it is sufficient to demonstrate the polynomial convexity of the graph of $\overline{P}$ over $\Gamma_{\mathbb{G}_{2}}$. To accomplish this, it is enough to prove that the graph of $\overline{P\circ \Pi}$ over $\mathbb{T}^2$ is polynomially convex. Following the notation in \Cref{R:Descriptn_Hull_Jimbo}, we have
	$h(z)=\frac{1}{z_1}+\frac{1}{z_2}-\frac{2}{z_1z_2}.$ 
	
	\begin{align*}
		\triangle({z})=&
		\begin{vmatrix} 
			\frac{\partial (P\circ\Pi)}{\partial{z_1}}& \frac{\partial (P\circ\Pi)}{\partial z_2}\\[1.5ex]
			\frac{\partial h}{\partial{{z_1}}}& \frac{\partial h}{\partial{z_2}}\\[1.5ex]
		\end{vmatrix}
		=
		\begin{vmatrix} 
			1-2z_2& 1-2z_1\\[1.5ex]
			\frac{-1}{z^2_1}+\frac{2}{z^2_1z_2}& \frac{-1}{z^2_2}+\frac{2}{z^2_2z_1}\\[1.5ex]
		\end{vmatrix}\\
		=&\frac{1}{z^2_{1}z^{2}_{2}}(z_1+z_2-2-2z_1z_2)(z_2-z_1).
	\end{align*}
	\noindent We define $q_{1}:=z_1+z_2-2-2z_1z_2,~q_2:=z_{2}-z_{1}.$ and $Z_{j}=\{z\in \cplx^{2}:q_{j}(z)=0\},j=1,2.$ Therefore, 	
	\begin{align*}
		\Sigma&=\left\{(z,w)\in \overline{\mathbb{D}}^{2}\setminus (L\cup \mathbb{T}^2): \triangle_{(z,w)}=0\right\}\\
		&=\left\{(z,w)\in \overline{\mathbb{D}}^{2}\setminus (L\cup \mathbb{T}^2)\right\}\cap[\cup^{2}_{j=1} Z_{j}],
	\end{align*}
	and	
	\begin{align*}
		X&=\left\{z\in \overline{\mathbb{D}}^{2}\setminus (L\cup \mathbb{T}^2): \overline{(P\circ \Pi)(z)}=h(z)\right\}\\
		&=\left\{z\in \overline{\mathbb{D}}^{2}\setminus (L\cup \mathbb{T}^2): \overline{z_1+z_2-2z_1z_2}=\frac{1}{z_1}+\frac{1}{z_2}-\frac{2}{z_1z_2}\right\}.
	\end{align*}
	
	\noindent Here $Q_{j}=Z_{j}\cap \mathbb{{T}}^2.$ We now claim that
	\begin{align*}
		\widehat{Q_1}&=\{z\in \cplx^2:z_1+z_2-2z_1z_2-2=0,|z_1|= 1,|z_2|= 1\}= Q_1.
	\end{align*}
	
	\noindent Clearly, $\widehat{Q_1}\subset \{z\in \cplx^2:z_1+z_2-2z_1z_2-2=0,|z_1|\le 1,|z_2|\le 1\}.$ Let $(\alpha,\beta)\in \{z\in \cplx^2:z_1+z_2-2z_1z_2-2=0,|z_1|\le 1,|z_2|\le 1\}\setminus Q_1.$ First, we assume that $|\beta|<1.$ Since $\alpha+\beta-2\alpha\beta-2=0,$ hence 
	\begin{align}\label{E:Pt_OutsideBidsc}
		|2-\alpha|=	|\beta||1-2\alpha|<|1-2\alpha|.
	\end{align}
	Let $\alpha=u+iv.$ Then from (\ref{E:Pt_OutsideBidsc}), we get that 
	\begin{align*}
		&(2-u)^2+v^2<(1-2u)^2+4v^2\\
		&\implies 4+u^2+v^2-4u<1+4(u^2+v^2)-4u\\
		&\implies u^2+v^2>1 \text{ i.e., } |\alpha|>1.
	\end{align*}
	Hence, we conclude that $(\alpha, \beta) \notin \widehat{Q_1}.$ In the case where $|\alpha|<1,$ we can similarly demonstrate that $|\beta|>1,$ leading to the same conclusion, $(\alpha, \beta) \notin \widehat{Q_1}.$ As a result, we establish that $Q_{1}$ is polynomially convex.
	
	\smallskip
	
	\noindent Furthermore, consider $\widehat{Q_2}=\{z\in \mathbb{C}^2:z_1=z_2,|z_1|\le 1\}\ne Q_2.$ Notably, $(\frac{1}{2},\frac{1}{2})\in \widehat{Q_{2}}\setminus (\mathbb{T}^2\cup L),$ while $(\frac{1}{2},\frac{1}{2})\notin X.$ Hence, by \Cref{R:Descriptn_Hull_Jimbo}, we can deduce that:
	
	\begin{align*}
		\widehat{\Gr_{\overline{P\circ \Pi}}(\mathbb{T}^2)}=\Gr_{\overline{P\circ \Pi}}(\mathbb{T}^2).
	\end{align*}
	
	\noindent This implies:
	
	\begin{align*}
		\widehat{\Gr_{\overline{P}}(\Gamma_{\mathbb{G}{2}})}=\Psi\left(\Gr_{\overline{P\circ \Pi}}(\mathbb{T}^2)\right)=\Gr_{\overline{P}}(\Gamma_{\mathbb{G}_{2}}).
	\end{align*}

	\begin{example}
		Let $p_{1}(z_1,z_2)=2z_1+z^2_2,~~p_{2}(z_1,z_2)=z_1-z^2_2,~P(z_1,z_2)=z_1-z_{2}$ and $\phi(z_1,z_2)=(p_{1}(z_1,z_2),p_{2}(z_1,z_2)).$ Therefore $\Omega=\phi(\mathbb{D}^2).$ Then $[z_1,z_2,\overline{P};\Gamma_{\Omega}]=\smoo(\Gamma_{\Omega}).$
	\end{example}	
	
	\noindent {\bf Explanation} 
	According to \Cref{L: Approx_Cont_Func_SymmBidsc}, it follows that $[z_1,z_2,\overline{P};\Gamma_{\Omega}]=\smoo(\Gamma_{\Omega})$ if, and only if, $\Gr_{\overline{P}}(\Gamma_{\Omega})$ exhibits polynomial convexity. Furthermore, the polynomial convexity of $\Gr_{\overline{P}}(\Gamma_{\Omega})$ is equivalent to the polynomial convexity of $\Gr_{\overline{P\circ \phi}}(\mathbb{T}^2)$.
	
	Here $\overline{P\circ \phi}=\overline{z_1+2z^2_2}=\frac{1}{z_1}+\frac{2}{z^2_2}=:h(z)$ on $\mathbb{T}^2.$  
	
	\begin{align*}
		\triangle({z})&=
		\begin{vmatrix} 
			\frac{\partial (P\circ\phi)}{\partial{z_1}}& \frac{\partial (P\circ\phi)}{\partial z_2}\\[1.5ex]
			\frac{\partial h}{\partial{{z_1}}}& \frac{\partial h}{\partial{z_2}}\\[1.5ex]
		\end{vmatrix}
		=
		\begin{vmatrix} 
			1& 4z_2\\[1.5ex]
			\frac{-1}{z^2_1}& \frac{-4}{z^3_2}\\[1.5ex]
		\end{vmatrix}
		=\frac{1}{z^2_{1}z^{3}_{2}}(z_1+z^2_2)(z^2_2-z_1).
	\end{align*}
	\noindent We define $q_{1}:=z_1+z^2_2,~q_2:=z^2_2-z_1,$ and $Z_{j}=\{z\in \cplx^{2}:q_{j}(z)=0\},j=1,2.$ Therefore, 	
	\begin{align*}
		\Sigma&=\left\{(z,w)\in \overline{\mathbb{D}}^{2}\setminus (L\cup \mathbb{T}^2): \triangle_{(z,w)}=0\right\}\\
		&=\left\{(z,w)\in \overline{\mathbb{D}}^{2}\setminus (L\cup \mathbb{T}^2)\right\}\cap[\cup^{2}_{j=1} Z_{j}],
	\end{align*}
	and	
	\begin{align*}
		X&=\left\{z\in \overline{\mathbb{D}}^{2}\setminus (L\cup \mathbb{T}^2): \overline{(P\circ \phi)(z)}=h(z)\right\}\\
		&=\left\{z\in \overline{\mathbb{D}}^{2}\setminus (L\cup \mathbb{T}^2): \overline{z_1+2z^2_2}=\frac{1}{z_1}+\frac{2}{z^2_2}\right\}.
	\end{align*}

	\noindent Here $Q_{j}=Z_{j}\cap \mathbb{{T}}^2.$ Clearly,
	\begin{align*}
		\widehat{Q_1}&=\{z\in \cplx^2:z_1+z^2_2=0, |z_1|\le 1,|z_2|\le 1\}\ne Q_1,~\text{ and}\\
		\widehat{Q_2}&=\{z\in \cplx^2:z^2_2-z_1=0, |z_1|\le 1,|z_2|\le 1\}\ne Q_2.
	\end{align*}

	\noindent It is easy to see that $\widehat{Q_{j}}\setminus (\mathbb{T}^2\cup L)\nsubseteq X$  for $j=1,2.$ Therefore, by a \Cref{R:Descriptn_Hull_Jimbo}, we get that	
	\begin{align*}
		\widehat{\Gr_{\overline{P\circ \phi}}(\mathbb{T}^2)}=\Gr_{\overline{P\circ \phi}}(\mathbb{T}^2).
	\end{align*}
	Hence 
	\begin{align*}
		\widehat{\Gr_{\overline{P}}(\Gamma_{\Omega})}&=\Psi\left(\Gr_{\overline{P\circ \phi}}(\mathbb{T}^2) \right)=	\Gr_{\overline{P}}(\Gamma_{\Omega}).
	\end{align*}

	\begin{example}
		Let $p_{1}(z_1,z_2)=z_1+z_2,~~p_{2}(z_1,z_2)=z^2_1+z^2_2,~P(z_1,z_2)=z^2_1+z_{2}$ and $\phi(z_1,z_2)=(p_{1}(z_1,z_2),p_{2}(z_1,z_2)).$ Therefore $\Omega=\phi(\mathbb{D}^2).$ Then $[z_1,z_2,\overline{P};\Gamma_{\Omega}]\ne \smoo(\Gamma_{\Omega}).$
	\end{example}	
	
	\noindent {\bf Explanation} 
	Based on \Cref{L: Approx_Cont_Func_SymmBidsc}, we can assert that $[z_1,z_2,\overline{P};\Gamma_{\Omega}]\ne \smoo(\Gamma_{\Omega})$ if, and only if, $\Gr_{\overline{P}}(\Gamma_{\Omega})$ lacks polynomial convexity. Furthermore, $\Gr_{\overline{P}}(\Gamma_{\Omega})$ possesses polynomial convexity if, and only if, $\Gr_{\overline{P\circ \phi}}(\mathbb{T}^2)$ is polynomially convex. Therefore, it is enough to show that $\Gr_{\overline{P\circ \phi}}(\mathbb{T}^2)$ is not polynomailly convex.
	
	Here $P\circ \phi=2(z^2_1+z_1z_2+z^2_2).$ Hence, 
	\begin{align*}
		\overline{P\circ \phi}=\overline{2(z^2_1+z_1z_2+z^2_2)}=2\left(\frac{1}{z^2_1}+\frac{1}{z^2_2}+\frac{1}{z_1z_2}\right)=:h(z) \text{ on } \mathbb{T}^2.  
	\end{align*}

	\begin{align*}
		\triangle({z})&=
		\begin{vmatrix} 
			\frac{\partial (P\circ\phi)}{\partial{z_1}}& \frac{\partial (P\circ\phi)}{\partial z_2}\\[1.5ex]
			\frac{\partial h}{\partial{{z_1}}}& \frac{\partial h}{\partial{z_2}}\\[1.5ex]
		\end{vmatrix}
		=
		\begin{vmatrix} 
			2(2z_1+z_2)& 2(2z_2+z_1)\\[1.5ex]
			2(\frac{-2}{z^3_1}-\frac{-1}{z^2_1z_2})& 2(\frac{-2}{z^3_2}-\frac{-1}{z_1z^2_2})\\[1.5ex]
		\end{vmatrix}\\[1.5ex]
		&=\frac{8\alpha^{-1}}{z^3_{1}z^{3}_{2}}(z_1+z_2)(z_2-z_1)(z_1-\alpha z_2)(z_2-\alpha z_1), \text{ where } \alpha=e^{\frac{2\pi i}{3}}.
	\end{align*}
	
	\noindent We define $q_{1}:=z_1+z_2,~q_2:=z_2-z_1,~,q_3=z_1-\alpha z_2,~q_{4}=z_2-\alpha z_1,$ and $Z_{j}=\{z\in \cplx^{2}:q_{j}(z)=0\},j=1,2,3,4.$ Therefore, 	
	\begin{align*}
		\Sigma&=\left\{z\in \overline{\mathbb{D}}^{2}\setminus (L\cup \mathbb{T}^2): \triangle(z)=0\right\}\\
		&=\left\{z\in \overline{\mathbb{D}}^{2}\setminus (L\cup \mathbb{T}^2)\right\}\cap[\cup^{3}_{j=1} Z_{j}],
	\end{align*}
	and	
	\begin{align*}
		X&=\left\{z\in \overline{\mathbb{D}}^{2}\setminus (L\cup \mathbb{T}^2): \overline{(P\circ \phi)(z)}=h(z)\right\}\\
		&=\left\{z\in \overline{\mathbb{D}}^{2}\setminus (L\cup \mathbb{T}^2): \overline{2(z^2_1+z_1z_2+z^2_2)}=2\left(\frac{1}{z^2_1}+\frac{1}{z^2_2}+\frac{1}{z_1z_2}\right)\right\}.
	\end{align*}

	\noindent Here $Q_{j}=Z_{j}\cap \mathbb{{T}}^2.$ Clearly,
	\begin{align*}
		\widehat{Q_1}&=\{z\in \cplx^2:z_1+z_2=0, |z_1|\le 1,|z_2|\le 1\}\ne Q_1;\\
		\widehat{Q_2}&=\{z\in \cplx^2:z_2-z_1=0, |z_1|\le 1,|z_2|\le 1\}\ne Q_2;\\
		\widehat{Q_3}&=\{z\in \cplx^2:z_1-\alpha z_2=0, |z_1|\le 1,|z_2|\le 1\}\ne Q_3;\\
		\widehat{Q_4}&=\{z\in \cplx^2:z_2-\alpha z_1=0, |z_1|\le 1,|z_2|\le 1\}\ne Q_4.
	\end{align*}

	\smallskip
	
	\noindent Again $\widehat{Q_{j}}\setminus (\mathbb{T}^2\cup L)\nsubseteq X$  for $j=1,2,$ and $\widehat{Q_{j}}\setminus (\mathbb{T}^2\cup L)\subset X$ for $j=3,4.$ Therefore, by \Cref{R:Descriptn_Hull_Jimbo}, we get that	
	\begin{align*}
		\widehat{\Gr_{\overline{P\circ \phi}}(\mathbb{T}^2)}=\Gr_{\overline{P\circ \phi}}(\mathbb{T}^2)\cup \Gr_{\overline{P\circ \phi}}(\widehat{Q_{3}})\cup \Gr_{\overline{P\circ \phi}}(\widehat{Q_{4}}).
	\end{align*}
	Hence 
	\begin{align*}
		\widehat{\Gr_{\overline{P}}(\Gamma_{\Omega})}&=\Psi\left(\Gr_{\overline{P\circ \phi}}(\mathbb{T}^2) \right)=\Gr_{\overline{P}}(\Gamma_{\Omega})\cup \Psi\left(\Gr_{\overline{P\circ \phi}}(\widehat{Q_{3}})\right)\cup \Psi\left(\Gr_{\overline{P\circ \phi}}(\widehat{Q_{4}})\right).
	\end{align*}

	\noindent{\bf Acknowledgements.}
	We would like to express our sincere gratitude to Professor Franc Forstneri\v{c} for pointing out \Cref{R:Image_AlgVariety} in \cite{Cirka69} and showing us the proof of \Cref{L:Img_Analytic Variety}. The first named author was partially supported by a Matrics Research Grant (MTR/2017/000974) of SERB, Dept. of Science and Technology, Govt. of India, for the beginning of this work and is supported by a Core Research Grant (CRG/2022/003560) of SERB, Dept. of Science and Technology, Govt. of India,  for the later part of the work. The second named author's work received partial support from an INSPIRE Fellowship (IF 160487) provided by the Dept. of Science and Technology, Govt. of India, during the early stage of this work. Presently, this research is supported by a research grant from SERB (Grant No. CRG/2021/005884), Dept. of Science and Technology, Govt. of India.

	%\bibliographystyle{plain}
	%\bibliography{biblio.bib}

\end{document}